\documentclass[reqno, 12pt]{amsart}
\pdfoutput=1
\makeatletter
\let\origsection=\section \def\section{\@ifstar{\origsection*}{\mysection}} 
\def\mysection{\@startsection{section}{1}\z@{.7\linespacing\@plus\linespacing}{.5\linespacing}{\normalfont\scshape\centering\S}}
\makeatother        

\usepackage{amsmath,amssymb,amsthm}
\usepackage{mathrsfs}
\usepackage{mathabx}\changenotsign
\usepackage{dsfont}

\usepackage{todonotes}
\usepackage{verbatim}
 
\usepackage{graphicx}
\usepackage{xcolor}
\usepackage{tikz}
\usetikzlibrary{arrows}

\usepackage[backref]{hyperref}
\hypersetup{
    colorlinks,
    linkcolor={red!60!black},
    citecolor={green!60!black},
    urlcolor={blue!60!black},
    pdftitle={Disjoint dijoins for classes of dicuts in finite and infinite digraphs},
    pdfauthor={J. Pascal Gollin, Karl Heuer, and Konstantinos Stavropoulos}
}

\usepackage[open,openlevel=2,atend]{bookmark}

\usepackage[abbrev,msc-links,backrefs]{amsrefs} 
\usepackage{doi}

\renewcommand{\PrintDOI}[1]{\doi{#1}}

\usepackage[T1]{fontenc}
\usepackage{lmodern}
\usepackage[babel]{microtype}
\usepackage[english]{babel}

\usepackage{geometry}
\numberwithin{equation}{section}
\numberwithin{figure}{section}

\usepackage{enumitem}

\let\polishlcross=\l
\def\l{\ifmmode\ell\else\polishlcross\fi}

\def\paragraph#1{%
  \noindent\textbf{#1.}\enspace}

\let\emptyset=\varnothing
\let\setminus=\smallsetminus

\makeatletter
\def\moverlay{\mathpalette\mov@rlay}
\def\mov@rlay#1#2{\leavevmode\vtop{   \baselineskip\z@skip \lineskiplimit-\maxdimen
   \ialign{\hfil$\m@th#1##$\hfil\cr#2\crcr}}}
\newcommand{\charfusion}[3][\mathord]{
    #1{\ifx#1\mathop\vphantom{#2}\fi
        \mathpalette\mov@rlay{#2\cr#3}
      }
    \ifx#1\mathop\expandafter\displaylimits\fi}
\makeatother

\DeclareFontFamily{U}  {MnSymbolC}{}
\DeclareSymbolFont{MnSyC}         {U}  {MnSymbolC}{m}{n}
\DeclareFontShape{U}{MnSymbolC}{m}{n}{
    <-6>  MnSymbolC5
   <6-7>  MnSymbolC6
   <7-8>  MnSymbolC7
   <8-9>  MnSymbolC8
   <9-10> MnSymbolC9
  <10-12> MnSymbolC10
  <12->   MnSymbolC12}{}
\DeclareMathSymbol{\powerset}{\mathord}{MnSyC}{180}

\let\epsilon=\varepsilon

\let\rho=\varrho
\let\theta=\vartheta
\let\phi=\varphi

\newcommand{\nbd}{\nobreakdash-}

\theoremstyle{plain}
\newtheorem{thm}{Theorem}[section]
\newtheorem{theorem}[thm]{Theorem}
\newtheorem{lemma}[thm]{Lemma}
\newtheorem{corollary}[thm]{Corollary}
\newtheorem{proposition}[thm]{Proposition}

\newtheorem{conjecture}[thm]{Conjecture}
\newtheorem{question}[thm]{Question}
\newtheorem*{claim*}{Claim}
\newtheorem{claim}{Claim}[]
\newtheorem{thm-intro}{Theorem}[]
\newtheorem{conj-intro}[thm-intro]{Conjecture}
\newtheorem{question-intro}[thm-intro]{Question}
\newtheorem*{meta-question*}{Meta question}

\theoremstyle{definition}

\newtheorem{remark}[thm]{Remark}

\newtheorem{construction}[thm]{Construction}
\newtheorem{example}{Example}
\newtheorem*{example*}{Example}

\newcommand{\abs}[1]{\ensuremath{{\lvert {#1} \rvert}}}
\newcommand{\restricted}{\ensuremath{{\upharpoonright}}}

\newcommand{\setsep}{\ensuremath{{\,\colon\,}}}
\newcommand{\triangularprism}{\ensuremath{K_3 \,\Box\, K_2}}

\DeclareMathOperator{\insh}{in}
\DeclareMathOperator{\outsh}{out}

\DeclareMathOperator{\tr}{tr}

\begin{document}

\author[J.~P.~Gollin]{J.~Pascal Gollin}
\address{J. Pascal Gollin, Discrete Mathematics Group, Institute for Basic Science (IBS), 55 Expo-ro, Yuseong-gu, Daejeon, Korea, 34126}
\email{\tt pascalgollin@ibs.re.kr}
\thanks{The first author was supported by the Institute for Basic Science (IBS-R029-Y3).}

\author[K.~Heuer]{Karl Heuer}
\address{Karl Heuer, Institute of Software Engineering and Theoretical Computer Science, Technische Universit\"{a}t Berlin, Ernst-Reuter-Platz 7, 10587 Berlin, Germany}
\email{\tt karl.heuer@tu-berlin.de}
\thanks{The second author was supported by a postdoc fellowship of the German Academic Exchange Service (DAAD) and partly by the European Research Council (ERC) under the European Union's Horizon 2020 research and innovation programme (ERC consolidator grant DISTRUCT, agreement No.\ 648527)}

\author[K.~Stavropoulos]{Konstantinos Stavropoulos}
\address{Konstantinos Stavropoulos, Fachbereich Mathematik, Universit\"{a}t Hamburg, Bundes{-}stra{\ss}e 55, 20146 Hamburg, Germany}
\email{\tt konstantinos.stavropoulos@uni-hamburg.de}

\title[Disjoint dijoins for classes of dicuts in finite and infinite digraphs]{Disjoint dijoins for classes of dicuts \\ in finite and infinite digraphs}

\date{September 8, 2021}



\begin{abstract}
    A \emph{dicut} in a directed graph is a cut for which all of its edges are directed to a common side of the cut. 
    A famous theorem of Lucchesi and Younger states 
    that in every finite digraph 
    the least size of a set of edges meeting every non-empty dicut 
    equals the maximum number of disjoint dicuts in that digraph. 
    Such sets are called \emph{dijoins}. 
    Woodall conjectured a dual statement. 
    He asked whether the maximum number of disjoint dijoins in a directed graph 
    equals the minimum size of a non-empty dicut. 
    
    We study a modification of this question where we restrict our attention to certain classes of non-empty dicuts, 
    i.e.~whether for a class~$\mathfrak{B}$ of dicuts of a directed graph the maximum number of disjoint sets of edges meeting every dicut in~$\mathfrak{B}$ equals the size of a minimum dicut in~$\mathfrak{B}$. 
    In particular, we verify this questions 
    for nested classes of finite dicuts, 
    for the class of dicuts of minimum size, 
    and for classes of infinite dibonds, 
    and we investigate how this generalised setting relates to a capacitated version of this question. 
\end{abstract}

\maketitle

\section{Introduction}
\label{sec:intro}

In this paper we consider directed graphs, which we briefly denote as \emph{digraphs}. 
A \emph{dicut} in a digraph is a cut for which all of its edges are directed to a common side of the cut. 
A famous theorem of Lucchesi and Younger~\cite{lucc-young_paper} states that in every finite digraph the least size of a set of edges meeting every non-empty dicut equals the maximum number of disjoint dicuts in that digraph. 
Such sets of edges are called \emph{dijoins}. 

Woodall conjectured the following in some sense dual statement, where the roles of the minimum and maximum are reversed. 

\begin{conj-intro}
    [Woodall 1976 \cite{woodall}]
    \label{conj:woodall}
    The size of a smallest 
    non-empty 
    dicut in a 
    finite 
    digraph~$D$ is equal to the size of the largest set of disjoint dijoins of~$D$.
\end{conj-intro}

This conjecture is a long-standing open question in this area and is included in a list of important conjectures compiled by Cornu\'{e}jols~\cite{Cornuejols:book}. 

Not much is known in general about this conjecture. 
It is easy to find two disjoint dijoins in a bridgeless weakly connected digraph~$D$: 
just consider an orientation of the underlying undirected multigraph that yields a strongly connected digraph, which exists by a theorem of Robbins~\cite{strongly-connected-orientation}; 
then the two desired dijoins of~$D$ are the set of edges which agrees with this orientation and the set of edges which disagrees with this orientation. 
This observation was noted by Seymour and DeVos~\cite{open-problem-garden}. 

But beyond that, there is no known bound on the size of a smallest dicut that can guarantee the existence of even three disjoint dijoins~\cites{egres-forum,open-problem-garden}.

There are several partial results restricting the attention to digraphs with certain properties. 
Lee and Wakabayashi~\cite{LW:series-parallel} verified Conjecture~\ref{conj:woodall} for digraphs whose underlying multigraph is series-parallel, which was later improved by Lee and Williams~\cite{LW:no_K^5-e} proving the conjecture for digraphs whose underlying multigraph is planar and does not contain a triangular prism~$\triangularprism$ as a minor\footnotemark. 
Schrijver~\cite{Schrijver:source-sink} and independently Feofiloff and Younger~\cite{FY:source-sink} verified the conjecture for \emph{source-sink connected} digraphs, 
i.e.~digraphs where from every source there exists a directed path to every sink. 
For additional partial results see~\cite{meszaros:partition-connected}. 

\footnotetext{Note that they proved the dual statement about feedback arc sets in planar digraphs without a minor isomorphic to~${K_5 - e}$, the graph obtained from~$K_5$ by deleting one of its edges, which is the planar dual of the triangular prism~$\triangularprism$.}

Thomassen~\cite{thomassen:perso} showed with some tournament on~$15$ vertices that a dual version regarding directed cycles and disjoint \emph{feedback arc sets}, i.e.\ sets of edges meeting every directed cycle, fails for non-planar digraphs. 
For planar digraphs, these questions are obviously equivalent and still open. 

A capacitated version of Woodall's Conjecture (cf.~Section~\ref{sec:capacity}), conjectured by Edmonds and Giles~\cite{EG:capacity} was proven to be false by Schrijver~\cite{Schrijver:counterexample}. 
Although false in general, the conjecture of Edmonds and Giles has been verified in some special cases.
In particular, the works by Lee and Wakabayashi~\cite{LW:series-parallel}, Lee and Williams~\cite{LW:no_K^5-e}, and Feofiloff and Younger~\cite{FY:source-sink} 
are actually about the conjecture of Edmonds and Giles and obtain corresponding results about Woodall's Conjecture as corollaries.
For more research regarding this line of work, including a study of the structure of possible counterexamples, see~\cites{CG:more-counterexamples, WG:even-more-counterexamples, CEKSS:dijoins_recent, shepherd_vetta}. 

\medskip

Instead of focusing on specific classes of digraphs, one other possible avenue to explore Conjecture~\ref{conj:woodall} is to restrict the attention not to all dicuts, but to some specific classes of dicuts of a digraph. 
In~\cites{inf-LY:1, inf-LY:2}, the first two authors considered a similar approach regarding classes of dicuts in their attempt of generalising the theorem of Lucchesi and Younger to infinite digraphs. 

For a digraph~$D$ and a non-empty class of non-empty dicuts~$\mathfrak{B}$ of~$D$, a set of edges meeting every dicut in~$\mathfrak{B}$ is a \emph{$\mathfrak{B}${\nbd}dijoin}. 
Using this terminology, a natural modification of Conjecture~\ref{conj:woodall} is the following question. 

\begin{question-intro}
    \label{quest:main}
    For which digraphs~$D$ 
    and classes of dicuts~$\mathfrak{B}$ of~$D$ 
    is the size of a smallest dicut in~$\mathfrak{B}$ equal to the size of a largest set of disjoint $\mathfrak{B}${\nbd}dijoins of~$D$?
\end{question-intro}

\medskip

A \emph{dibond} of~$D$ is a minimal non-empty dicut. 
We say a class~$\mathfrak{B}$ of dicuts of a digraph~$D$ is \emph{dibond-closed} if every dibond which is contained in some dicut in~$\mathfrak{B}$ is contained in~$\mathfrak{B}$ as well. 
Note that whenever~$\mathfrak{B}$ is a dibond-closed class of dicuts, then Question~\ref{quest:main} for this class is equivalent to the question for the class of all dibonds contained in~$\mathfrak{B}$. 
Hence in the setting of Question~\ref{quest:main}, Conjecture~\ref{conj:woodall} translates into the question for the class~$\mathfrak{B}_\textnormal{dibond}$ of all dibonds of~$D$. 

For classes of dicuts which are not dibond-closed in digraphs that are not necessarily weakly connected, this question relates to the capacitated version of Question~\ref{quest:main} for classes of dibonds. 
We will investigate this connection in Section~\ref{sec:capacity}. 

\medskip

In this paper, we will give positive answers for Question~\ref{quest:main} for several classes of dicuts. 
Two dicuts are \emph{nested} if some side of one of them is a subset of some side of the other. 
A set of dicuts is \emph{nested} if the dicuts in that set are pairwise nested. 
We prove a result 
for nested classes of finite dicuts, 
using the machinery developed by Berge~\cite{berge} for transversal packings in balanced hypergraphs. 

\begin{thm-intro}
    \label{thm:nested-woodall}
    Let~$D$ be a digraph 
    and~$\mathfrak{B}$ be a nested class of finite dicuts of~$D$. 
    Then the size of a smallest 
    dicut in~$\mathfrak{B}$ is equal to the size of a largest set of disjoint $\mathfrak{B}${\nbd}dijoins of~$D$.
\end{thm-intro}

Another interesting class is the class~$\mathfrak{B}_\textnormal{min}$ of dicuts of minimum size. 
An inductive construction allows us to reduce the following theorem to Theorem~\ref{thm:nested-woodall}. 

\begin{thm-intro}
    \label{thm:mini-woodall}
    The size of a smallest 
    dicut in a digraph~$D$ 
    that contains finite dicuts 
    is equal to the size of largest set of disjoint $\mathfrak{B}_\textnormal{min}${\nbd}dijoins of~$D$, 
    where~$\mathfrak{B}_\textnormal{min}$ denotes the class of dicuts of~$D$ of minimum size. 
\end{thm-intro}

The parts of Theorems~\ref{thm:nested-woodall} and~\ref{thm:mini-woodall} regarding infinite digraphs are proved using the compactness principle in combinatorics. 
We will also use that technique to prove a finitary version of the results of Lee and Williams~\cite{LW:no_K^5-e} and Feofiloff and Younger~\cite{FY:source-sink} for infinite digraphs only considering dicuts of finite size (capacity). 

Finally, we verify a cardinality version for classes of infinite dibonds of Question~\ref{quest:main}. 
This proof uses a transfinite recursion to construct the dijoins in the case where the dibonds have the same cardinality as the order of the digraph. 
To complete the proof, we use the concept of bond-faithful decompositions due to Laviolette~\cite{laviolette}. 

\begin{thm-intro}
    \label{thm:infinite-cardinality-woodall}
    Let~$D$ be a digraph and~$\mathfrak{B}$ be a class of infinite dibonds of~$D$.
    The size of a smallest dibond in~$\mathfrak{B}$ 
    is equal to the size of largest set of disjoint $\mathfrak{B}${\nbd}dijoins of~$D$.
\end{thm-intro}

\medskip

This paper is structured as follows. 
After introducing some terminology in Section~\ref{sec:prelims}, 
we discuss the capacitated version of Question~\ref{quest:main} and its connection to classes of dicuts which are not dibond-closed in Section~\ref{sec:capacity}. 
In Section~\ref{sec:dicut-hypergraph}, we look at transversal packings of balanced hypergraphs and prove  Theorem~\ref{thm:nested-woodall}.
Section~\ref{sec:mini-woodall} is dedicated to prove Theorem~\ref{thm:mini-woodall} for finite digraphs. 
We turn our attention to infinite digraphs in Section~\ref{sec:infinite}. 
After showing that a more structural generalisation of Conjecture~\ref{conj:woodall} to infinite digraphs fails (see Example~\ref{ex:no-erdos-menger-woodall}), in Subsection~\ref{subsec:compactness} we use the compactness principle in combinatorics to finish the proof of Theorem~\ref{thm:mini-woodall} for infinite digraphs and prove a finitary version of the results of Lee and Williams~\cite{LW:no_K^5-e} and Feofiloff and Younger~\cite{FY:source-sink}. 
Finally, in Subsection~\ref{subsec:infinite-dibonds} we prove Theorem~\ref{thm:infinite-cardinality-woodall} and give an example that a capacitated version of that theorem fails.

\section{Preliminaries}
\label{sec:prelims}

For general facts and notation for graphs we refer the reader to~\cite{diestel}, for digraphs in particular to~\cite{bang-jensen}, and for hypergraphs to~\cite{berge}. 
For some set theoretic background, including ordinals, cardinals and transfinite induction, we refer the reader to~\cite{jech}. 

\subsection{Digraphs}
\ 

Let~$D$ be a digraph with vertex set~${V(D)}$ and edge set~${E(D)}$. 
We allow~$D$ to have parallel edges, but may assume for most purposes in this paper that~$D$ does not contain any loops. 
We view the edges of~$D$ as ordered pairs~${(u,v)}$ of vertices~${u, v \in V(D)}$ and shall write~${uv}$ instead of~${(u, v)}$, although this might not uniquely determine an edge if~$D$ contains parallel edges. 
We say~$D$ is \emph{simple} if it does not contain parallel edges (or loops). 
For an edge~${uv \in E(D)}$ we furthermore call the vertex~$u$ as the \emph{tail} of~${uv}$ and~$v$ as the \emph{head} of~$uv$. 

A digraph~$D$ is \emph{weakly connected} if its underlying undirected (multi-)graph is connected. 
The components of the underlying undirected (multi-)graph are the \emph{weak components of~$D$}. 

For two sets~${X, Y \subseteq V(D)}$ of vertices we define~${E_D(X, Y) \subseteq E(D)}$ as the set of those edges 
that have their head in~${X \setminus Y}$ and their tail in~${Y \setminus X}$, 
or their head in~${Y \setminus X}$ and their tail in~${X \setminus Y}$. 

We consider a \emph{bipartition of~$V(D)$} to be an ordered pair~${(X,Y)}$ for which~${X \cap Y = \emptyset}$ and~${X \cup Y = V(D)}$. 
We call~${(X,Y)}$ \emph{trivial} if either~$X$ or~$Y$ is empty. 

A set~$B$ of edges of~$D$ is a \emph{cut} of~$D$ if there is a non-trivial bipartition~${(X,Y)}$ of~$V(D)$ such that~${B = E_D(X,Y)}$. 
We call~${(X,Y)}$ a \emph{representation} of~$B$ (or say~${(X,Y)}$ \emph{represents}~$B$), 
and we refer to~$X$ and~$Y$ as the \emph{sides} of the representation 
(or the \emph{sides} of the cut, if the representation is inferred from context). 
Note that a cut of a weakly connected digraph has 
up to the ordering of the pair 
a unique representation. 
For a set~$\mathcal{B}$ of cuts, a set~$\mathcal{R}$ of bipartitions of~${V(D)}$ \emph{represents}~$\mathcal{B}$ if for each~${B \in \mathcal{B}}$ there is a bipartition~${(X,Y) \in \mathcal{R}}$ that represents~$B$. 

We define a partial order on the set of bipartitions of~${V(D)}$ by 
\[
    (X,Y) \leq (X',Y') \quad \textnormal{ if and only if } \quad X \subseteq X' \textnormal{ and } Y \supseteq Y'. 
\]
Two bipartitions~${(X,Y)}$ and~${(X',Y')}$ if~${V(D)}$ are \emph{nested} 
if one of~${(X,Y)}$,~${(Y,X)}$ is $\leq${\nbd}comparable with one of~${(X',Y')}$,~${(Y',X')}$. 
Note that~${(X,Y) \leq (X',Y')}$, if and only if~${(Y',X') \leq (Y,X)}$. 
A set~$\mathcal{R}$ of bipartitions of~${V(D)}$ is nested if its elements are pairwise nested. 

Two cuts~$B_1$ and~$B_2$ are \emph{nested} if they can be represented by nested bipartitions of~${V(D)}$. 
Moreover, we call a set (or sequence)~$\mathcal{B}$ of cuts of~$D$ nested if there is a nested set of bipartitions that represents~$\mathcal{B}$. 
Note that in a digraph which is not weakly connected, a set of pairwise nested cuts is not necessarily nested itself. 
If two cuts of~$D$ are not nested, we call them \emph{crossing}. 

A \emph{bond} is a minimal non-empty cut (with respect to the subset relation). 
Note that if~$D$ is weakly connected, then a cut~${B}$ of~$D$ represented by~${(X,Y)}$ is a bond, if and only if the induced subdigraphs~${D[X]}$ and~${D[Y]}$ are weakly connected digraphs. 

We call a cut~$B$ \emph{directed}, or briefly a \emph{dicut}, if all its have their head in one common side of the cut. 
A bond that is also a dicut is called a \emph{dibond}. 

If a bipartition~${(X,Y)}$ of~${V(D)}$ represents a non-empty dicut, we call the side of the representation that contains the heads of the edges of the dicut the \emph{in-shore~${\insh_D(X,Y)}$} of the representation, and we call the side of the representation that contains the tails of the edges of the dicut the \emph{out-shore~${\outsh_D(X,Y)}$} of the representation. 
If the representation of a dicut~$B$ is clear from the context (for example if~$D$ is weakly connected) we will speak of the \emph{in-shore~${\insh_D(B)}$} of~$B$, or \emph{out-shore~${\outsh_D(B)}$ of~$B$}. 

A cut or dicut \emph{separates} two vertices (or two sets of vertices) if there is a representation for which they are contained in different sides. 
Similarly,~$B$ separates two sets of vertices, if they are contained in different sides of some representation of~$B$. 

\medskip

Let~$\mathfrak{B}$ be a class of dibonds of~$D$. 
A \emph{$\mathfrak{B}${\nbd}dijoin} is a set of edges meeting every dibond in~$\mathfrak{B}$. 
We say~$D$ is \emph{$\mathfrak{B}${\nbd}woodall} if the size of a smallest dibond in~$\mathfrak{B}$ is equal to the size of a largest set of disjoint $\mathfrak{B}${\nbd}dijoins of~$D$.

\subsection{Hypergraphs}
\ 

A \emph{hypergraph}~$\mathcal{H}$ is a tuple~$(X,H)$ consisting of a set~$X$ and a set~$H$ of subsets of~$X$. 
The set~$X$ is called the \emph{ground set} of~$\mathcal{H}$ and the elements of~$H$ are called \emph{hyperedges}. 

A set~${T \subseteq X}$ is a \emph{transversal of~$\mathcal{H}$} if~${T \cap h}$ is non-empty for every hyperedge~${h \in H}$. 

For a set~${Y \subseteq X}$ we define the \emph{subhypergraph of~$\mathcal{H}$ induced on~$Y$} as
\[
    {\mathcal{H}[Y] := \big( Y, \{ h \cap Y \setsep h \cap Y \neq \emptyset, h \in H \} \big)}. 
\]
For a subset~${H' \subseteq H}$, the hypergraph~${\mathcal{H}(H') = (X, H')}$ is the \emph{partial hypergraph of~$\mathcal{H}$ generated by~$H'$}. 
A \emph{partial subhypergraph of~$\mathcal{H}$} is a partial hypergraph of an induced subhypergraph. 
Lastly, for a set~${Y \subseteq X}$ we call for~${H \restricted Y := \{ h \in H \setsep H \subseteq Y \}}$ the partial subhypergraph~${\mathcal{H} \restricted Y := (Y, H \restricted Y)}$ 
the \emph{restriction of~$\mathcal{H}$ to~$Y$}. 

For a positive integer~$k$, a map~${c \colon\, X \to [k]}$ is a \emph{$k${\nbd}colouring of~$\mathcal{H}$} if no hyperedge of size at least~$2$ is monochromatic, i.e.~$\abs{c(h)} > 1$ for each~${h \in H}$ with~${\abs{h} > 1}$. 
If there is a $k${\nbd}colouring of~$\mathcal{H}$, then we say~$\mathcal{H}$ is $k${\nbd}colourable.

\section{On the capacitated version}
\label{sec:capacity}

A map~${c \colon E(D) \to K}$, where~$K$ denotes a set of cardinals, is called a \emph{capacity} of~$D$. 
If~${K = \mathbb{N}}$ (the set of non-negative integers), then we say~$c$ is \emph{finitary}. 
We say a capacity~${c \colon E(D) \to \{0,1\}}$ is \emph{simple}. 

Given a capacity~$c$ and a subset~${A \subseteq E(D)}$, then by~$c(A)$ we denote the cardinality of the disjoint union of~${\{c(a) \setsep a \in A\}}$ (which is considered to be the cardinal sum of the cardinals). 
If this cardinal is finite, then we say~$A$ has \emph{finite capacity}. 
For a finitary capacity~$c$ and a finite~${A \subseteq E(D)}$ this coincides to the integer~${\sum_{e \in A} c(e)}$, and for a simple capacity~$c$ this coincides with the cardinal~${\abs{A \cap c^{-1}(1)}}$.

Given a capacity~$c$ of~$D$, 
a dibond~${B \in \mathfrak{B}}$ is \emph{$c${\nbd}cheapest in~$\mathfrak{B}$} if~${c(B)}$ is minimum among all dibonds in~$\mathfrak{B}$, 
and a set of dijoins is \emph{$c${\nbd}disjoint} if each edge~${e \in E(D)}$ is contained in at most~${c(e)}$ dijoins of that set. 
We say~$D$ is \emph{$\mathfrak{B}${\nbd}woodall with respect to~$c$} if the size of a $c${\nbd}cheapest dibond in~$\mathfrak{B}$ is equal to the size of a largest set of $c${\nbd}disjoint $\mathfrak{B}${\nbd}dijoins. 
This leads to the following capacitated version of Question~\ref{quest:main}.

\begin{question}
    \label{quest:capacitated}
    Which digraphs~$D$ with capacity~$c$ 
    and classes of dicuts~$\mathfrak{B}$ 
    are $\mathfrak{B}${\nbd}woodall with respect to~$c$? 
\end{question}

In this setting, the conjecture of Edmonds and Giles~\cite{EG:capacity} (which was proven to be false by Schrijver~\cite{Schrijver:counterexample}, cf.~Figure~\ref{fig:schrijver}) translates to the statement that every finite digraph~$D$ with capacity~$c$ of~$D$ is~$\mathfrak{B}_{\textnormal{dibond}}${\nbd}woodall with respect to~$c$, where~$\mathfrak{B}_{\textnormal{dibond}}$ denotes the class of all finite dibonds of~$D$. 

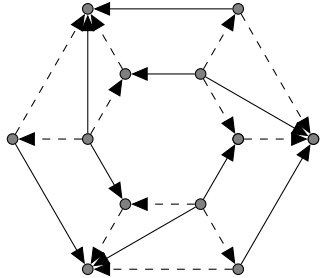
\begin{figure}[htbp]
    \centering
    \begin{tikzpicture}
        [scale=1]
        \tikzset{vertex/.style = {circle, draw, fill=black!50, inner sep=0pt, minimum width=4pt}}
        \tikzset{edge0/.style = {-triangle 45, dashed}}
        \tikzset{edge1/.style = {-triangle 45, black}}
        
        \foreach \x in {0,...,6} {
            \node [vertex] (v\x) at (60*\x:2) {};
            \node [vertex] (w\x) at (60*\x:1) {};
        }
        \foreach \x [evaluate={\y = int(\x+1);},evaluate={\z = int(\x+2);}] in {0,2,4}{
            \draw [edge0] (w\x) -- (v\x) {};
            \draw [edge0] (w\y) -- (v\y) {};
            \draw [edge1] (w\y) -- (v\x) {};
            \foreach \v in {v,w} {
                \draw [edge0] (\v\y) -- (\v\x) {};
                \draw [edge1] (\v\y) -- (\v\z) {};
            }
        }
    \end{tikzpicture}
    \caption{Schrijver's counterexample to the conjecture of Edmonds and Giles. The dashed edges have capacity~$0$, and the solid edges have capacity~$1$.}
    \label{fig:schrijver}
\end{figure}

Given a capacity~$c$ of a digraph~$D$ one can replace each edge of~$e$ of positive capacity by a set of~$\abs{c(e)}$ many distinct parallel edges. 
Defining the capacity of each of these newly created edges as~$1$ we obtain a digraph with a simple capacity, and it is not hard to see that the question whether for a class~$\mathfrak{B}$ of dicuts is $\mathfrak{B}${\nbd}woodall with respect to~$c$ is equivalent to the corresponding question for the `corresponding' class of dicuts of~$D'$. 
In fact, the key feature of any known counterexample of the conjecture of Edmonds and Giles is to assign capacity~$0$ to some edges. 
Going one step further by deleting the edges of capacity~$0$ allows us equivalently talk about Question~\ref{quest:capacitated} in the setting of Question~\ref{quest:main}, as the following construction and proposition shows. 

\begin{construction}
    \label{construction:capacity2dicuts}
    Given a digraph~$D$ with a capacity~$c$ of~$D$ we define a digraph~$\hat{D}$. 
    
    Let~$\hat{D}$ be the digraph on~$V(D)$ obtained by replacing each edge~$e$ of~$D$ with~$c(D)$ many distinct edges~${\{ e_\alpha \setsep 0 \leq \alpha < c(D) \}}$, each of which has the same head and tail as~$e$. 
    Note that any edge of capacity~$0$ is deleted in this process. 
    
    For each dicut~$B$ of~$D$ represented by~$(X,Y)$, we define a corresponding dicut~$\hat{B}$ of~$\hat{D}$ as~${\hat{B} := E_{\hat{D}}(X,Y)}$. 
    It is not hard to see that this is indeed well-defined 
    and that the capacity of~$B$ equals the size of~$\hat{B}$. 
    
    Moreover, for a class~$\mathfrak{B}$ of dicuts of~$D$ we define~$\hat{\mathfrak{B}}$ as~${\{ \hat{B} \setsep B \in \mathfrak{B}\}}$. 
\end{construction}

\begin{proposition}
    \label{prop:capacity}
    Let~$D$ be a digraph, let~$c$ be a capacity of~$D$, and let~$\mathfrak{B}$ be a class of dicuts of~$D$. 
    Then~$D$ is $\mathfrak{B}${\nbd}woodall with respect to~$c$ if and only if~$\hat{D}$ is $\hat{\mathfrak{B}}${\nbd}woodall. 
\end{proposition}

\begin{proof} 
    For a set~$\hat{A}$ of edges of~$\hat{D}$ we call the set~${\tr(\hat{A}) := \{ e \in E(D) \setsep e_\alpha \in \hat{A}\}}$ the \emph{trace} of~$\hat{A}$. 
    For a set~$A$ of edges of~$D$ and a set~$\hat{A}$ of edges of~$\hat{D}$, we say the pair~$(A, \hat{A})$ is \emph{compatible} if~${A = \tr(\hat{A})}$. 
    Note that for each set~$A$ of edges~$e$ of~$D$ such that~${c(e) > 0}$ for all~${e \in A}$, there is a set of edges of~$\hat{D}$ such that~$(A, \hat{A})$ is compatible.
    
    For a compatible pair~${(F,\hat{F})}$ it is easy to see that~$F$ is a $\mathfrak{B}${\nbd}dijoin of~$D$ if and only if~$\hat{F}$ is a $\hat{\mathfrak{B}}${\nbd}dijoin of~$\hat{D}$. 
    Moreover, if~$\hat{\mathcal{F}}$ is a set of disjoint $\hat{\mathfrak{B}}${\nbd}dijoins of~$\hat{D}$, then~${\{ \tr(\hat{F}) \setsep \hat{F} \in \hat{\mathcal{F}} \}}$ is $c${\nbd}disjoint. 
    
    Given a $c${\nbd}disjoint set~${\{ F_\alpha \setsep \alpha \in \kappa \}}$ of $\mathfrak{B}${\nbd}dijoins of~$D$ for some cardinal~$\kappa$ we define a set of disjoint $\hat{\mathfrak{B}}${\nbd}dijoins iteratively as follows. 
    In step~$\alpha$ we construct~$\hat{F}_\alpha$ by taking for each~${e \in F_\alpha}$ the edge~$e_\beta$ for the smallest ordinal~$\beta$ such that~${e_\beta \notin \bigcup \{ \hat{F}_\gamma \setsep \gamma < \alpha \}}$. 
    By the assumption that~$e$ is contained in at most~${c(e)}$ many of the dijoins, this construction is well-defined and~$(F_\alpha, \hat{F}_\alpha)$ is compatible. 
    Hence we have the desired equivalence. 
\end{proof}

Recall that we say a class~$\mathfrak{B}$ of dicuts of a digraph~$D$ is \emph{dibond-closed} if every dibond which is contained in some dicut in~$\mathfrak{B}$ is contained in~$\mathfrak{B}$ as well. 
In the context of Question~\ref{quest:main} it is quite natural to consider classes which are dibond-closed. 
Indeed, whenever we consider Question~\ref{quest:main} for a dibond-closed class~$\mathfrak{B}$ we can equivalently consider the question for the class of dibonds in~$\mathfrak{B}$.
But by considering the construction from Proposition~\ref{prop:capacity} we may destroy weak connectivity by deleting the capacity~$0$ edges. 
Therefore, a dibond in~$\mathfrak{B}$ may correspond to a dicut in~$\hat{\mathfrak{B}}$ which is not a dibond. 
In this way, Question~\ref{quest:capacitated} for dibonds can be thought of as a special case of Question~\ref{quest:main}, albeit for a slightly modified digraph. 

In fact, the reverse is also true, as the following construction shows. 

\begin{construction}
    \label{construction:dicuts2capacity}
    Given a digraph~$D$ we define a digraph~$\tilde{D}$ with a capacity~$\tilde{c}$ of~$\tilde{D}$. 
    
    For every subset~${S \subseteq V(D)}$ we add a distinct vertex~$v_S$ and edges~${v_S s}$ for each~${s \in S}$. 
    Now we define a capacity~${\tilde{c}}$ on~$\tilde{D}$ by setting~${c(e) = 1}$ if~$e$ is an edge of~$D$ and~$0$ otherwise. 
    
    For each dicut~$B$ of~$D$ represented by~${(X,Y)}$, we define a corresponding dibond~$\tilde{B}$ of~$\tilde{D}$ as follows. 
    We define~${X_B \subseteq V(\tilde{D})}$ as the union of~$\insh_D(X,Y)$ with~${\{ v_S \setsep S \subseteq \insh_D(X,Y) \}}$. 
    Then we define~$\tilde{B}$ as the set of ingoing edges of~$X_B$ in~$\tilde{D}$, i.e.~${E_D(V(D) \setminus X_B, X_B)}$. 
    It is not hard to see that this is indeed well-defined 
    and that the size of~$B$ equals the capacity of~$\tilde{B}$ since the edges of capacity~$1$ that~$\tilde{B}$ contains are precisely the edges in~$B$. 
    
    Moreover, for a class~$\mathfrak{B}$ of dicuts of~$D$ we define~$\tilde{\mathfrak{B}}$ as~${\{ \tilde{B} \setsep B \in \mathfrak{B}\}}$. 
\end{construction}

\begin{proposition}
    \label{prop:capacity2}
    Let~$D$ be a digraph and let~$\mathfrak{B}$ be a class of dicuts of~$D$. 
    Then~$D$ is $\mathfrak{B}${\nbd}woodall if and only if~$\tilde{D}$ is $\tilde{\mathfrak{B}}${\nbd}woodall with respect to~$\tilde{c}$. 
\end{proposition}

\begin{proof}
    Since by construction~${B \subseteq \tilde{B}}$ for every dicut~$B$ of~$D$ we get that every $\mathfrak{B}${\nbd}dijoin of~$D$ is a $\tilde{\mathfrak{B}}${\nbd}dijoin of~$\tilde{D}$. 
    Vice versa, every $\tilde{\mathfrak{B}}${\nbd}dijoin of~$\tilde{D}$ that does not contain any capacity~$0$ edges is a $\mathfrak{B}${\nbd}dijoin of~$D$. 
    Since the capacity~$\tilde{c}$ is simple, the proposition immediately follows. 
\end{proof}

Therefore Question~\ref{quest:capacitated} for classes of dibonds can really be thought of Question~\ref{quest:main} for non dibond-closed classes of dicuts. 

In fact this observation will yield capacitated versions of many of our main results as it is easy to observe that if a class~$\mathfrak{B}$ of dibonds of a digraph~$D$ is nested, then so are~$\hat{\mathfrak{B}}$ and~$\tilde{\mathfrak{B}}$.

\section{The dicut hypergraph}
\label{sec:dicut-hypergraph}

Let~$D$ be a digraph and let~$\mathfrak{B}$ be a class of dicuts of~$D$.
Then~${\mathcal{H}(D, \mathfrak{B}) := \big( E(D) , \mathfrak{B} \big)}$ is the \emph{$\mathfrak{B}${\nbd}dicut hypergraph} of~$D$. 
Note that a transversal of~${\mathcal{H}(D, \mathfrak{B})}$ is a $\mathfrak{B}${\nbd}dijoin. 
If~$\mathfrak{B}$ is the set of all dicuts of~$D$, then we denote by~${\mathcal{H}(D)}$ the \emph{dicut hypergraph}~${\mathcal{H}(D, \mathfrak{B})}$.

\subsection{Transversal packings of hypergraphs}
\ 

Let~$\mathcal{H} = (X, H)$ be a hypergraph. 
A tuple~${(x_1, h_1, x_2, h_2, \dots, x_n, h_n, x_{n + 1})}$ is a \emph{Berge-cycle} of~$\mathcal{H}$
if
\begin{enumerate}
    \item ${x_1, \dots, x_n \in X}$ are distinct; 
    \item ${h_1, \dots, h_n \in H}$ are distinct; 
    \item ${x_{n+1} = x_1}$; and
    \item ${x_i, x_{i+1} \in h_i \in H}$ for all~${i \in [n]}$. 
\end{enumerate}
The \emph{length} of the Berge-cycle is~$n$. 
A Berge-cycle of odd length is called \emph{odd}. 
We call a Berge-cycle \emph{improper} if some hyperedge~$h_i$ contains some~$x_j$ for~${j \notin \{i, i+1\}}$. 

We call~$\mathcal{H}$ \emph{balanced} if every odd Berge-cycle is improper.
Balanced hypergraphs are one type of generalisation of bipartite graphs. 
By a theorem of Berge, finite balanced hypergraphs contain~$k$ pairwise disjoint transversals for~$k$ being the minimum size of a hyperedge. 
For the sake of completeness we will include a proof of this theorem. 

\begin{theorem}
    [Berge {\cite{berge}*{Corollary~2 of Section~5.3}}]
    \label{thm:berge}
    Every finite balanced hypergraph \linebreak
    ${\mathcal{H} = (X, H)}$ contains~${k := \min_{h \in H} \abs{h}}$ disjoint transversals of~$\mathcal{H}$.
\end{theorem}

In order to provide a proof of this theorem, we will use the following characterisation of balanced hypergraphs. 

\begin{theorem}
    [Berge {\cite{berge}*{Theorem~7 of Section~5.3}}]
    \label{thm:berge1}
    A finite hypergraph~$\mathcal{H}$ is balanced, if and only if every induced subhypergraph of~$\mathcal{H}$ is $2${\nbd}colourable.
\end{theorem}

\begin{proof}
    If~$\mathcal{H}$ contains an odd Berge-cycle~${(x_1, h_1, x_2, h_2, \dots, x_n, h_n, x_{n + 1})}$ which is not improper, then the hypergraph induced on~${\{x_i \setsep i \in [n]\}}$ contains the edges of an odd cycle, and hence is not $2${\nbd}colourable. 
    
    For the other direction it is enough to show that a finite balanced hypergraph is $2${\nbd}colourable. 
    Suppose for a contradiction that~${\mathcal{H} = (X,H)}$ is a counterexample with a ground set~$X$ of minimum size. 
    From the minimality we can deduce that each~${x \in X}$ is contained in at least two hyperedges of size~$2$ and hence the subgraph~$G$ of~$\mathcal{H}$ containing all hyperedges of size~$2$ has minimum degree~$2$. 
    Since~$\mathcal{H}$ is balanced,~$G$ is bipartite. 
    Let~${x \in X}$ be such that it is no cut-vertex of~$G$. 
    The hypergraph induced on~${X \setminus \{x\}}$ has a $2${\nbd}colouring by the minimality assumption. 
    Since~$G$ is bipartite and~$x$ is not a cut-vertex, the neighbourhood of~$x$ in~$G$ is monochromatic. 
    Hence we can extend the $2${\nbd}colouring to~$\mathcal{H}$, contradicting that~$\mathcal{H}$ is a counterexample. 
\end{proof}

Given this theorem, we can prove Theorem~\ref{thm:berge}. 

\begin{proof}[Proof of Theorem~\ref{thm:berge}]
    For~${k = 1}$ the statement is obvious and if~${k = 2}$, the statement follows directly from Theorem~\ref{thm:berge1} since each colour class of a proper $2${\nbd}colouring is a transversal. 
    If~${k > 2}$, consider a $k${\nbd}colouring~$c$ for which the sum~${\sum_{h \in H} \abs{c(h)}}$ is as large as possible. 
    Note that if this sum equals~${k \abs{H}}$, then each colour class of~$c$ is a transversal. 
    So suppose for a contradiction that the sum is smaller. 
    Then there is a hyperedge~$h_0$ with~${\abs{c(h_0)} < k}$. 
    Since by assumption every hyperedge has size at least~$k$, there is a colour~$p$ appearing twice on~$h_0$ as well as a colour~$q$ not appearing on~$h_0$. 
    Consider the subhypergraph induced on the colour classes~$S_p$ and~$S_q$ of these two colours. 
    By Theorem~\ref{thm:berge1}, this hypergraph has a $2${\nbd}colouring~$c'$. 
    But then 
    \[
        \hat{c}(x) := 
        \begin{cases}
            c(x) & \textnormal{ if } x \notin S_p \cup S_q\\
            c'(x) & \textnormal{ if } x \in S_p \cup S_q
        \end{cases}
    \]
    defines a $k${\nbd}colouring for which~${\sum_{h \in H} \abs{\hat{c}(h)} > \sum_{h \in H} \abs{c(h)}}$, a contradiction. 
\end{proof}

The technique of the compactness principle in combinatorics allows us to push these results about finite hypergraphs to infinite hypergraphs of \emph{finite character}, i.e.~hypergraphs in which every hyperedge is finite.
We omit stating the compactness principle here but refer to~\cite{diestel}*{Appendix~A}. 

\begin{lemma}
    \label{lem:compatness-hypergraphs}
    Let~${\mathcal{H} = (X, H)}$ be a hypergraph of finite character 
    such that for each finite~${Y \subseteq X}$ there is a finite~${\overline{Y} \subseteq X}$ containing~$Y$ such that~${\mathcal{H} \restricted \overline{Y}}$ 
    contains~${\min_{h \in H \restricted \overline{Y}} \abs{h}}$ disjoint transversals of~${\mathcal{H} \restricted \overline{Y}}$. 
    Then~$\mathcal{H}$ contains~${k := \min_{h \in H} \abs{h}}$ disjoint transversals of~$\mathcal{H}$. 
\end{lemma}

\begin{proof}
    We construct the~$k$ disjoint transversals of~$\mathcal{H}$ via compactness. 
    Given sets~$Z$ and~$Y$ with~${Z \subseteq Y \subseteq X}$, note that for a transversal~$T$ of~${\mathcal{H} \restricted Y}$ the set~${T \cap Z}$ is a transversal of~${\mathcal{H} \restricted Z}$. 
    Hence, by the assumptions of this lemma and the compactness principle there is a set~${\{ T_i \setsep i \in [k] \}}$ of subsets of~$X$ such that for each finite~$Y \subseteq X$ the set~${\{ T_i \cap \overline{Y} \setsep i \in [k] \}}$ is a set of~$k$ disjoint transversals of~${\mathcal{H} \restricted \overline{Y}}$. 
    Hence~${T_i \cap T_j = \emptyset}$ for~${i \neq j}$. 
    Moreover, each~$T_i$ meets any~${h \in H}$ since~${T_i \cap \overline{h}}$ is a transversal of~${\mathcal{H} \restricted \overline{h}}$ which is a finite subhypergraph of~$\mathcal{H}$ since~$\mathcal{H}$ has finite character. 
    Therefore,~${\{ T_i \setsep i \in [k] \}}$ is as desired. 
\end{proof}

Since every partial subhypergraph of a balanced hypergraph is again balanced (cf.~\cite{berge}*{Proposition~1 of Section~5.3}, we obtain the following corollary of Lemma~\ref{lem:compatness-hypergraphs} and Theorem~\ref{thm:berge}. 

\begin{corollary}
    \label{cor:berge-finite-character}
    Every balanced hypergraph~${\mathcal{H} = (X, H)}$ of finite character 
    contains \linebreak 
    ${k := \min_{h \in H} \abs{h}}$ disjoint transversals of~$\mathcal{H}$. \qed
\end{corollary}

\subsection{Applications to the dicut hypergraph}
\ 

We now turn our attention back to the dicut hypergraph and will show that the dicut hypergraph for a nested class of dicuts is balanced. 
Note that the following lemma holds for infinite digraphs as well. 

\begin{lemma}
    \label{lem:balanced}
    For any digraph~$D$, 
    every odd Berge-cycle~${(e_1, B_1, e_2, B_2, \hdots, e_{n}, B_{n}, e_1)}$ of the dicut hypergraph~${\mathcal{H}(D)}$ of~$D$ 
    for which the set~${\{ B_i \setsep i \in [n]\}}$ is nested, is improper. 
\end{lemma}

\begin{proof}
    Let~${(e_1, B_1, e_2, B_2, \hdots, e_{n}, B_{n}, e_1)}$ be and odd Berge-cycle 
    and let~${\{ (X_i, Y_i) \setsep i \in [n]\}}$ be a nested set of bipartitions of~${V(D)}$ representing~${\{ B_i \setsep i \in [n]\}}$ such that~${(X_i,Y_i)}$ represents~$B_i$ and~${\insh_D(X_i,Y_i) = Y_i}$ for all~${i \in [n]}$. 
    
    By setting~$B_0 := B_n$,~$X_0 := X_n$ and~${Y_0 := Y_n}$, since~${e_{i+1} \in B_i \cap B_{i+1}}$ for all~${0 \leq i < n}$, 
    we get either~${(X_i,Y_i) \leq (X_{i+1},Y_{i+1})}$ or~${(X_{i+1},Y_{i+1}) \leq (X_i, Y_i)}$. 
    While~$n$ is odd, these two possibilities cannot occur in an alternating fashion throughout the whole cycle.
    Hence we may assume without loss of generality that either~${(X_n,Y_n) \leq (X_1,Y_1) \leq (X_2,Y_2)}$ or~${(X_2,Y_2) \leq (X_1,Y_1) \leq (X_n,Y_n)}$. 
    We continue the argument with the former inequality, the other case is symmetric. 
    
    Consider the set~${I}$ of all~${i \in [n]}$ for which 
    either~${(X_1,Y_1) \leq (X_i,Y_i)}$ or~${(Y_i,X_i) \leq (X_1,Y_1)}$. 
    Since~${2 \in I}$ and~${n \notin I}$, 
    there is an integer~$j$ with~${2 \leq j < n}$ with~${j \in I}$ and~${j+1 \notin I}$. 
    And since~${B_j \cap B_{j+1}}$ is nonempty and hence~${(X_j,Y_j)}$ and~${(X_{j+1},Y_{j+1})}$ are $\leq${\nbd}comparable, we get 
    either~${(Y_j,X_j) \leq (X_1,Y_1) \leq (Y_{j+1},X_{j+1})}$, 
    or~${(X_{j+1},Y_{j+1}) \leq (X_1,Y_1) \leq (X_j,Y_j)}$. 
    
    Note that the first case is not possible since~$e_{j+1}$ would be an edge with tail in~$Y_1$ and head in~$X_1$, contradicting that~$Y_1$ is the in-shore of~$B_1$. 
    Hence,~${(X_{j+1},Y_{j+1}) \leq (X_1,Y_1) \leq (X_j,Y_j)}$ and~${e_{j+1} \in B_1 \cap B_j \cap B_{j+1}}$, proving that~${(e_1, B_1, e_2, B_2, \hdots, e_{n}, B_{n}, e_1)}$ is improper. 
\end{proof}

Hence, if~$\mathfrak{B}$ is a nested class of dicuts of~$D$, Lemma~\ref{lem:balanced} shows that~${\mathcal{H}(D, \mathfrak{B})}$ is balanced, and with Theorem~\ref{thm:berge} we get Theorem~\ref{thm:nested-woodall} for finite digraphs, 
and together with Lemma~\ref{lem:compatness-hypergraphs} we can complete the proof.

\setcounter{thm-intro}{2}
\begin{thm-intro}
    Let~$D$ be a digraph 
    and~$\mathfrak{B}$ be a nested class of finite dicuts of~$D$. 
    Then~$D$ is $\mathfrak{B}${\nbd}woodall. 
\end{thm-intro}

\begin{proof}
    Let~$k$ denote the size of a smallest dicut in~$\mathfrak{B}$. 
    By Lemma~\ref{lem:balanced}, the $\mathfrak{B}${\nbd}dicut hypergraph~${\mathcal{H} := \mathcal{H}(D,\mathfrak{B})}$ is balanced and by assumption has finite character. 
    Each finite restriction of~$\mathcal{H}$ whose set of hyperedges is non-empty is balanced and hence contains~$k$ disjoint transversals by Theorem~\ref{thm:berge}. 
    The theorem follows from Lemma~\ref{lem:compatness-hypergraphs}.
\end{proof}

With the observations of Section~\ref{sec:capacity}, we obtain the capacitated version of this theorem. 

\begin{corollary}
    \label{cor:nested-woodall-cap}
    Let~$D$ be a digraph with a capacity~${c}$ 
    and let~$\mathfrak{B}$ be a nested class of dibonds of~$D$ of finite capacity. 
    Then~$D$ is $\mathfrak{B}${\nbd}woodall with respect to~$c$. 
    \qed
\end{corollary}

A dicut of~$D$ is \emph{atomic} if it has a representation in which one the sides contains only a single vertex, i.e.~a source or a sink. 
It is easily verified that a set of atomic dicuts is nested. 
Hence, we get the following corollary. 

\begin{corollary}
    \label{cor:atomic-woodall}
    Let~$D$ be a digraph (with a capacity~${c}$)
    and let~$\mathfrak{B}$ be a class of atomic dicuts of~$D$ of finite size (capacity). 
    Then~$D$ is $\mathfrak{B}${\nbd}woodall (with respect to~$c$). 
    \qed
\end{corollary}

\section{Minimum size dicuts and disjoint mini-dijoins}
\label{sec:mini-woodall}

Let~$D$ be a 
digraph 
and let $\mathfrak{B}$ be a class of dicuts of~$D$. 
We say~$\mathfrak{B}$ is \emph{corner-closed} if for each non-empty~${B, B' \in \mathfrak{B}}$ which are crossing and represented by~$(X,Y)$ and~$(X',Y')$, respectively, the dicuts
\[
    {B \wedge B' := {E}_D( \outsh_D(X,Y) \cup \outsh_D(X',Y'), \insh_D(X,Y) \cap \insh_D(X',Y') )} 
\]
and 
\[
    {B \vee B' := {E}_D( \outsh_D(X,Y) \cap \outsh_D(X',Y'), \insh_D(X,Y) \cup \insh_D(X',Y') )}
\]
are in~$\mathfrak{B}$. 
Note that it is easy to see that~${B \wedge B'}$ and~${B \vee B'}$ are indeed 
non-empty
dicuts whose definition does not depend on the choice of the representations of~$B$ and~$B'$. 
In particular, for any representations~${(X,Y)}$ and~${(X',Y')}$ of~$B$ and~$B'$, respectively, with~${\insh_D(X,Y) = Y}$ and~${\insh_D(X',Y') = Y'}$, the bipartition~$(X \cup X', Y \cap Y')$ represents~${B \wedge B'}$ and the bipartition~${(X \cap X', Y \cup Y')}$ represents~${B \vee B'}$. 

Moreover, we observe that for the digraphs~$\hat{D}$ and~$\tilde{D}$ from Constructions~\ref{construction:capacity2dicuts} and~\ref{construction:dicuts2capacity}, respectively, 
for dicuts~$B$ and~$B'$ of~$D$, we obtain
\[
    \hat{B} \wedge \hat{B'} = \widehat{B \wedge B'}, \
    \hat{B} \vee \hat{B'} = \widehat{B \vee B'}, \ 
    \tilde{B} \wedge \tilde{B'} = \widetilde{B \wedge B'}, \textnormal { and }
    \tilde{B} \vee \tilde{B'} = \widetilde{B \vee B'}.
\]
In particular, a class~$\mathfrak{B}$ of dicuts of~$D$ is corner-closed, if and only if~$\hat{\mathfrak{B}}$ is corner-closed, if and only if~$\tilde{\mathfrak{B}}$ is corner-closed. 

\begin{remark}
    \label{rem:colouring}
    Similarly to the proof of Theorem~\ref{thm:berge}, we can consider a set~$\{ F_i \setsep i \in [m]\}$ of disjoint $\mathfrak{B}${\nbd}dijoins of a digraph~$D$ as a partial colouring of the edges of~$D$ where an edge~$e$ is coloured with the colour~$i$ if and only if~${e \in F_i}$. 
    We call such a colouring~${f \colon \bigcup \{ F_i \setsep i \in [m] \} \to [m]}$ a \emph{$\mathfrak{B}${\nbd}woodall colouring of~$D$}. 
    
    Since each~$F_i$ is a $\mathfrak{B}${\nbd}dijoin, we obtain that each dicut~${B \in \mathfrak{B}}$ is coloured with every colour. 
    Note that if~${\lvert B \rvert = m}$, then~$B$ is necessarily \emph{colourful}, i.e.~$B$ contains every colour. 
\end{remark}

\begin{theorem}
    \label{thm:m-woddall}
    Let~$D$ be a finite digraph 
    (with capacity~${c}$), 
    let~$m$ be a positive integer 
    and let~$\mathfrak{B}$ denote a corner-closed class of dicuts of~$D$ all of size~$m$ 
    (capacity~$m$). 
    Then~$D$ is $\mathfrak{B}${\nbd}woodall (with respect to~$c$). 
\end{theorem}

\begin{proof}
    The capacitated version of this theorem follows from the non-capacitated version by the observations of both Section~\ref{sec:capacity} and above. 
    We prove the non-capacitated version by induction on the number of non-atomic dicuts in~$\mathfrak{B}$. 
    If~$\mathfrak{B}$ contains only atomic dicuts, then the statement follows from Corollary~\ref{cor:atomic-woodall}. 
    
    Otherwise, let~${B \in \mathfrak{B}}$ be non-atomic represented by~$(X,Y)$ with~${\insh_D(X,Y) = Y}$. 
    Consider the digraph~$D_1$ obtained by identifying all vertices 
    in~$Y$ to a single vertex (and deleting loops, afterwards) 
    with~$\mathfrak{B}_1$ being the class of dicuts in~$\mathfrak{B}$ that are dicuts of~$D_1$, 
    as well as the digraph~$D_2$ obtained by identifying all vertices 
    in~${X}$ to a single vertex (and deleting loops, afterwards) 
    with~$\mathfrak{B}_2$ being the class of dicuts in~$\mathfrak{B}$ that are dicuts of~$D_2$. 
    By construction,~${E(D_1) \cap E(D_2) = B}$. 
    Note that in both~$\mathfrak{B}_1$ and~$\mathfrak{B}_2$ the number of non-atomic dicuts strictly decreased since~$B$ is atomic in both~$D_1$ and~$D_2$ and each non-atomic dicut of~$D_1$ or~$D_2$ is non-atomic in~$D$ as well. 
    By induction, for~${j \in \{1, 2\}}$ there are sets~${\{ F_i^j \setsep i \in [m] \}}$ of disjoint $\mathfrak{B}_j${\nbd}dijoins of~$D_j$. 
    Note that since~${\lvert B \rvert = m}$, for each~${e \in B}$ there is a unique~${i_e \in [m]}$ and a unique~${j_e \in [m]}$ such that~${e \in F_{i_e}^1 \cap F_{j_e}^2}$. 
    We claim that~${\{ F_{i_e}^1 \cup F_{j_e}^2 \setsep e \in B \}}$ is a set of disjoint $\mathfrak{B}${\nbd}dijoins. 
    As in Remark~\ref{rem:colouring}, we consider these edges sets as partial colourings of~${E(D)}$ with colours~$B$. 
    
    The fact that the sets are pairwise disjoint follows from the observation that they are the union of disjoint dijoins of~$D_1$ and~$D_2$ respectively and that~${E(D_1) \cap E(D_2) = B}$. 
    
    Note that any bipartition~$(X',Y')$ of~${V(D)}$ which is nested with~${(X,Y)}$ naturally defines a bipartition of~${V(D_1)}$ and~${V(D_2)}$, one of which is trivial and the other representing a cut of the respective digraph which equals~${E_D(X',Y')}$. 
    Let~${B' \in \mathfrak{B}}$ be represented by~${(X',Y')}$ with~${\insh_D(X',Y') = Y'}$ and assume~${B' \notin \mathfrak{B}_1 \cup \mathfrak{B}_2}$. 
    Consider the corners~${B \wedge B'}$ and~${B \vee B'}$ represented by~${(X \cup X', Y \cap Y')}$ and~${(X \cap X', Y \cup Y')}$, respectively. 
    Since~${B' \notin \mathfrak{B}_1 \cup \mathfrak{B}_2}$, both of the corners are non-empty, 
    and since~$\mathfrak{B}$ is corner-closed, they are in~$\mathfrak{B}$ as well. 
    Furthermore, since both~${(X \cup X', Y \cap Y)}$ and~${(X \cap X', Y \cup Y')}$ are nested with~${(X,Y)}$, we obtain that~${B \vee B' \in \mathfrak{B}_1}$ and~${B \wedge B' \in \mathfrak{B}_2}$. 
    Note that since both~${B \wedge B'}$ and~$B$ are colourful by assumption, 
    every colour that appears on~${B \cap E(D[X'])}$ 
    also appears on~${B' \cap E(D[Y'])}$, 
    and since~${B \vee B'}$ and~$B$ are colourful, 
    every colour that appears on~${B \cap E(D[Y'])}$ 
    also appears on~${B' \cap E(D[X])}$.
    Together with the colours appearing in~${B \cap B'}$ we deduce that~$B'$ is colourful as well, as desired. 
\end{proof}

\begin{lemma}
    \label{lem:mini-corner-closed}
    Let~$D$ be a digraph (with capacity~$c$) that contains a dicut of finite size (capacity). 
    The class~$\mathfrak{B}_\textnormal{min}$ of dicuts of~$D$ of minimum size (capacity) is corner-closed. 
\end{lemma}

\begin{proof}
    Let~${B_1, B_2 \in \mathfrak{B}_\textnormal{min}}$ be crossing. 
    Hence neither~${B_1 \wedge B_2}$ nor~${B_1 \vee B_2}$ are empty. 
    A simple double counting argument yields~${c(B_1) + c(B_2) = c(B_1 \wedge B_2) + c(B_1 \vee B_2)}$. 
    Therefore both~${B_1 \wedge B_2}$ and~${B_1 \vee B_2}$ are of capacity~${c(B_1)}$, since neither of them can be smaller. 
    Hence they are dicuts of minimum capacity and hence in~$\mathfrak{B}_\textnormal{min}$. 
\end{proof}

Now we can deduce the version of Theorem~\ref{thm:mini-woodall} for finite digraphs as a direct corollary of Theorem~\ref{thm:m-woddall} and Lemma~\ref{lem:mini-corner-closed}. 

\begin{theorem}
    \label{thm:mini-woodall-finite}
    Every finite 
    digraph~$D$ 
    (with finitary capacity~${c}$) 
    is $\mathfrak{B}_\textnormal{min}${\nbd}woodall 
    (with respect to~$c$), where~$\mathfrak{B}_\textnormal{min}$ denotes the class of dicuts of~$D$ of minimum size (capacity). 
    \qed
\end{theorem}

\section{Dijoins in infinite digraphs}
\label{sec:infinite}

First, we should ask what the `right' generalisation of Conjecture~\ref{conj:woodall} for infinite digraphs should be. 
Often a more structural generalisation of such min-max theorems yields a more meaningful result than a simple generalisation using cardinalities. 
For example, Erd\H{o}s conjectured and Aharoni and Berger~\cite{inf-menger} proved a structural generalisation of Menger's theorem, where in every graph we can simultaneously find a set~$\mathcal{P}$ of disjoint paths between two sets~$A$ and~$B$ of vertices and a set~$S$ of vertices separating~$A$ and~$B$ such that each path in~$\mathcal{P}$ contains precisely one vertex in~$S$ and~$S$ contains no vertices not included in some path in~$\mathcal{P}$. 

However, such an Erd\H{o}s-Menger-like structural generalisation of Conjecture~\ref{conj:woodall} fails in infinite digraphs, as we will illustrate in the following example. 

\begin{example} 
    \label{ex:no-erdos-menger-woodall}
    We give an example of an infinite digraph with no pair a dicut~$B$ and a set of disjoint dijoins~${\{ F_e \setsep e \in B\}}$ such that~${B \cap F_e = \{ e \}}$ for all~${e \in B}$.

    Consider the digraph~$D$ with vertex set~${V(D) := \mathbb{Z} \times \{ 1, -1 \}}$ and edge set 
    \[
        E(D) := \{ (z, i)(z+i,i) \setsep z \in \mathbb{Z}, i \in \{ 1, -1 \} \} \cup \{ (z, 1)(z, -1) \setsep z \in \mathbb{Z} \}, 
    \]
    as depicted in Figure~\ref{fig:counterexample}. 
    We call the edges~${(z,1)(z,-1)}$ of the second type the \emph{rungs} of~$D$.
    
    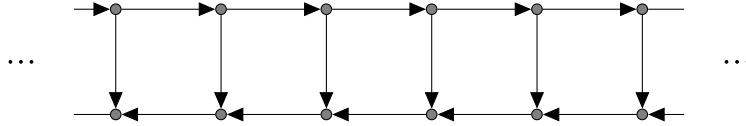
\begin{figure}[htbp]
        \centering
        
        \begin{tikzpicture}
            [scale=1.4]
            \tikzset{vertex/.style = {circle, draw, fill=black!50, inner sep=0pt, minimum width=4pt}}
            \tikzset{dot/.style = {circle, draw, fill=black, inner sep=0pt, minimum width=1pt}}
            \tikzset{arrow/.style = {-triangle 45}}
            
            \pgfmathtruncatemacro{\n}{6} 
            
            \node (v0) at (0.5,0) {};
            \node (v\the\numexpr \n+1 \relax) at (\n+0.5,0) {};
            \node (w0) at (0.5,1) {};
            \node (w\the\numexpr \n+1 \relax) at (\n+0.5,1) {};
            \foreach \x in {1,...,\n}{
                \node [vertex] (v\x) at (\x,0) {};
                \node [vertex] (w\x) at (\x,1) {};
            };
            \foreach \x [evaluate={\y = int(\x-1);}, evaluate={\z = int(\x+1);}] in {1,...,\n}{
                \draw [arrow] (w\y) edge (w\x) {};
                \draw [arrow] (v\z) edge (v\x) {};
                \draw [arrow] (w\x) edge (v\x) {};
            };
            \draw (v1) edge (v0) {};
            \draw (w\n) edge (w\the\numexpr \n+1 \relax) {};
            \foreach \x in {0,...,2} {
                \node [dot] at (0.1*\x,0.5) {};
                \node [dot] at (\n+0.8+0.1*\x,0.5) {};
            };
        \end{tikzpicture}
        
        \caption{A counterexample to an Erd\H{o}s-Menger-like structural generalisation of Conjecture~\ref{conj:woodall} to infinite digraphs.}
        \label{fig:counterexample}
    \end{figure}
    
    Note that~$D$ has no finite dicut and for each infinite dicut~$B$ there is an~${n \in \mathbb{Z}}$ such that~$B$ contains all rungs~${\{ (z, 1) (z, -1) \setsep z \leq n \}}$ up to~$n$. 
    Moreover, whenever~$F_1$ and~$F_2$ are two disjoint dijoins, then at least one of them contains infinitely many rungs of~${\{ (z, 1) (z, -1) \setsep z \leq n \}}$ up to any~${n \in \mathbb{Z}}$. 
    Hence every dicut meets 
    such a dijoin infinitely often. 
    Moreover, note that in this example, the set of dicuts is nested. 
\end{example}

In the light of this example, we can only hope for weaker generalisations to be possible. 

In Subsection~\ref{subsec:compactness}, we will consider Question~\ref{quest:main} for classes of dicuts of finite size (capacity), where the structural generalisation as hinted above is still equivalent to a comparison of cardinals. 
In this setting, we can extend the respective results from the finite case using the compactness principle (or more precisely, Lemma~\ref{lem:compatness-hypergraphs}). 
Such an approach will not work for classes of dicuts that contain dicuts of both finite and infinite size (capacity). 

In Subsection~\ref{subsec:infinite-dibonds}, we will consider Question~\ref{quest:main} for classes of infinite dibonds, where we will proof a cardinality-version of this question for any class of infinite dibonds. 
We will also show that even a cardinality-version for classes of infinite dicuts (and hence a capacitated version, cf.~Question~\ref{quest:capacitated}) can fail.

\subsection{Classes of finite dicuts in infinite digraphs}
\label{subsec:compactness}
\

Let~$D$ be a digraph. 
Given a set~$\mathcal{B}$ of dicuts of~$D$, we define an equivalence relation on~$V(D)$ 
by setting~${v \equiv_\mathcal{B} w}$ if and only if 
we cannot separate~$v$ from~$w$ by a dicut in~$\mathcal{B}$. 

It is easy to check that~$\equiv_\mathcal{B}$ indeed defines an equivalence relation. 
Let~${D{/}{\equiv_\mathcal{B}}}$ denote the digraph 
which is obtained from~$D$ by identifying the vertices in the same equivalence class of~$\equiv_\mathcal{B}$ and deleting loops. 
Note that~${D{/}{\equiv_\mathcal{B}}}$ does not contain any directed cycles. 
Given a capacity~$c$ of~$D$, we call the restriction of~$c$ to~${E(D{/}{\equiv_\mathcal{B}})}$ the \emph{capacity of~$D{/}{\equiv_\mathcal{B}}$ induced by~$c$}. 
We shall use the following observation from~\cite{inf-LY:1} about this digraph. 

\begin{proposition}
    \cite{inf-LY:1}*{Proposition~2.9(ii)}
    \label{prop:quotient}
    Let~$D$ be a digraph 
    and let~$\mathcal{B}$ be a set of dicuts of~$D$. 
    Then every dicut (or dibond, respectively) in~$\mathcal{B}$ of~$D$ is also a dicut (or dibond, respectively) of~${D{/}{\equiv_\mathcal{B}}}$. 
\end{proposition}

The main tool we use to extend results about Question~\ref{quest:main} from classes of finite graphs to finitary infinite versions is the following compactness-type lemma. 

\begin{lemma}
    \label{lem:compactness}
    Let~$D$ be a digraph (with capacity~$c$) 
    and let~$\mathfrak{B}$ be a class of dicuts of~$D$ of finite size (capacity). 
    Suppose that for every finite set~${\mathcal{B} \subseteq \mathfrak{B}}$ 
    there is a finite set~${\overline{\mathcal{B}} \subseteq \mathfrak{B}}$ containing~$\mathcal{B}$ such that the digraph~${D{/}\equiv_{\overline{\mathcal{B}}}}$ is $\overline{\mathcal{B}}${\nbd}woodall (with respect to the capacity induced by~$c$). 
    Then~$D$ is $\mathfrak{B}${\nbd}woodall (with respect to~$c$). 
\end{lemma}

\begin{proof}
    Let~$k$ denote the capacity of a $c${\nbd}cheapest dicut in~$\mathfrak{B}$ and let~${E' \subseteq E(D)}$ denote the set of all edges of~$D$ of positive capacity. 
    Consider 
    the $\hat{\mathfrak{B}}${\nbd}dicut hypergraph~${\mathcal{H} := \mathcal{H}(\hat{D},\hat{\mathfrak{B})}}$ for~$\hat{D}$ and~$\hat{\mathfrak{B}}$ as in Construction~\ref{construction:capacity2dicuts}. 
    It is easy to observe that~${\mathcal{H}}$ has finite character. 
    
    Moreover, note that for a finite set~${\mathcal{B} \subseteq \mathfrak{B}}$ containing a dicut of capacity~$k$ and for a~$\overline{\mathcal{B}}$ as in the assumption, 
    the hypergraph~${\mathcal{H}_\mathcal{B} := \mathcal{H}(\hat{D}{/}\equiv_{\widehat{\overline{\mathcal{B}}}}, \widehat{\overline{\mathcal{B}}})}$ has by Proposition~\ref{prop:capacity} and the assumption~$k$ disjoint transversals. 
    By construction~$\mathcal{H}_\mathcal{B}$ is a restriction of~$\mathcal{H}$ to a finite set. 
    Moreover, for each restriction of~$\mathcal{H}$ to a finite set~${Y \subseteq E(\hat{D})}$ there is a finite set~${\overline{Y} \subseteq E(\hat{D})}$ containing~$Y$ such that~${\mathcal{H} \restricted \overline{Y} = \mathcal{H}_{\mathcal{B}}}$ for some finite~${\mathcal{B} \subseteq \mathfrak{B}}$. 
    Hence the result follows from Lemma~\ref{lem:compatness-hypergraphs} and again Proposition~\ref{prop:capacity}. 
\end{proof}

\medskip

Lemma~\ref{lem:compactness} together with Theorem~\ref{thm:m-woddall} yield the following corollary.

\begin{corollary}
    \label{cor:m-woodall-infinite}
    Let~$D$ be a 
    digraph (with capacity~$c$), 
    let~$m$ be a positive integer 
    and let~$\mathfrak{B}$ denote a corner-closed class of dicuts of~$D$ all of 
    size~${m}$ (capacity~$m$). 
    Then~$D$ is $\mathfrak{B}${\nbd}woodall (with respect to~$c$). 
    \qed
\end{corollary}

Recall that by Lemma~\ref{lem:mini-corner-closed} and the observations in Section~\ref{sec:mini-woodall}, applying Construction~\ref{construction:capacity2dicuts} yields that~$\hat{\mathfrak{B}}_\textnormal{min}$ is corner-closed. 
Hence we deduce the following corollary (and hence Theorem~\ref{thm:mini-woodall}). 

\begin{corollary}
    \label{cor:mini-woodall-infinite}
    Let~$D$ be a 
    digraph (with capacity~$c$) that contains a dicut of finite size (capacity) and let~$\mathfrak{B}_\textnormal{min}$ be the set of dicuts of minimum size (capacity). 
    Then~$D$ is $\mathfrak{B}_\textnormal{min}${\nbd}woodall. 
    \qed
\end{corollary}

Lemma~\ref{lem:compactness} also yields the following corollary.

\begin{corollary}
    \label{cor:inf-finitary-woodall}
    If Conjecture~\ref{conj:woodall} is true for all weakly connected finite digraphs, 
    then every weakly connected digraph~$D$ is $\mathfrak{B}_\textnormal{fin}${\nbd}woodall for the class~$\mathfrak{B}_\textnormal{fin}$ of finite dicuts of~$D$. 
    \qed
\end{corollary}

\medskip

The \emph{triangular prism~\triangularprism} is the undirected graph with vertex set~${V(K_3) \times V(K_2)}$ and edges between~${(v_1,w_1)}$ and~${(v_2,w_2)}$ if and only if both~${v_1 v_2 \in E(K_3)}$ and~${w_1 w_2 \in E(K_2)}$. 
As mentioned in the introduction, the capacitated version Conjecture~\ref{conj:woodall} has been verified for planar digraphs with no minor isomorphic to the triangular prism~$\triangularprism$ by Lee and Williams. 

\begin{theorem}
    \cite{LW:no_K^5-e}
    \label{thm:finite_K_5-3}
    Every finite weakly connected digraph~$D$ 
    (with finitary capacity~$c$) 
    whose underlying multigraph is planar and contains no minor isomorphic to the triangular prism~$\triangularprism$ 
    is $\mathfrak{B}${\nbd}woodall 
    (with respect to~$c$) 
    for the class~$\mathfrak{B}$ of all dicuts of~$D$. 
\end{theorem}

In order to extend Theorem~\ref{thm:finite_K_5-3} to infinite digraphs, 
we begin by observing that each finite set~$\mathcal{B}$ of finite dicuts of~$D$ can be extended to a finite set~$\overline{\mathcal{B}}$ of finite dicuts of~$D$ such that the auxiliary graphs~${D{/}{\equiv_{\overline{\mathcal{B}}}}}$ is a minor\footnotemark\ of~$D$. 
If instead~$\mathcal{B}$ is a finite set of dicuts of finite capacity, then we can still extend~$\mathcal{B}$ to a finite set~$\overline{\mathcal{B}}$ and find a minor~$D^\ast$ of~$D$ which is $\mathcal{B}^\ast${\nbd}woodall with respect to the capacity obtained by restricting~$c$ to~${E(D^\ast)}$ for the set~$\mathcal{B}^\ast$ of all dicuts of~$D^\ast$, as we will establish in the following lemma. 

\footnotetext{We say a directed graph~$D$ is a minor of a directed graph~$D'$, if~$D$ can be obtained from~$D'$ by an arbitrary sequence of vertex-deletions, edge-deletions and edge-contractions. 
Note that in the literature these type of minors are also known as \emph{weak minors}.}

\begin{lemma}
    \label{lem:minorofaux2}
    Let~$D$ be a weakly connected digraph with capacity~$c$, let~$\mathcal{B}$ be a finite set of dicuts of~$D$ of finite capacity. 
    Then there is a finite set~${\overline{\mathcal{B}}}$ of dicuts of~$D$ of finite capacity with~${\mathcal{B} \subseteq \overline{\mathcal{B}}}$ 
    and a finite minor~$D^\ast$ of~$D$ such that for the set~$\mathcal{B}^\ast$ of dicuts of~$D^\ast$ 
    we have 
        ${\{ B \setminus c^{-1}(0) \colon B \in \mathcal{B}^\ast \} = 
        \{ B \setminus c^{-1}(0) \colon B \in \overline{\mathcal{B}} \}}$. 
\end{lemma}

\begin{proof}
    Let~${E' := \bigcup \mathcal{B}}$. 
    Let~$\mathcal{B}'$ be the set of all dicuts~$B$ of~$D$ with~${B \subseteq E'}$. 
    Note that each dicut in~$\mathcal{B}'$ has finite capacity. 
    Consider the digraph~$D'$ obtained from~$D$ by contracting~${E(D) \setminus E'}$. 
    Note that the sets of dicuts of~$D{/}{\equiv_{\mathcal{B}'}}$ and of~$D'$ coincide and both graphs have~$E'$ as their edge set. 
    We claim that~${D'}$ and~${D{/}{\equiv_{\mathcal{B}'}}}$ are isomorphic. 
    Suppose for a contradiction that two distinct vertices~$v$ and~$w$ of~$D'$ cannot be separated by a dicut in~$\mathcal{B}'$, and hence are contained in the same strong component of~$D'$. 
    Let~$P$ be a directed path from~$v$ to~$w$ in~$D'$. 
    As~$D{/}{\equiv_{\mathcal{B}'}}$ contains no directed cycles, each edge of~$P$ is a loop of~$D{/}{\equiv_{\mathcal{B}'}}$, contradicting its construction. 
    Conversely, any vertices~$v$ and~$w$ which are not equivalent with respect to~$\equiv_{\mathcal{B}'}$ do not lie in the same weak component of~${D - E'}$ and hence are not identified in~$D'$. 
    
    If~$\mathcal{B}'$ is finite, then~$D'$ has only finitely many vertices and the result follows with~${\overline{\mathcal{B}} := \mathcal{B}'}$ and the digraph~${D^\ast}$ obtained from~$D'$ by deleting for all pairs of vertices all but finitely many edges of capacity~$0$ between them. 
    So let us assume~$\mathcal{B}'$ and hence~$D'$ is infinite. 
    
    Let~$W$ the set of all vertices of~$D'$ that are incident with some edge of positive capacity. 
    If~$\mathcal{B}'$ contains a dicut with a representation~$(X,Y)$ for which~${W \subseteq Y}$, for some vertex~${v_0 \in X}$ we define~${W' := W \cup \{v_0\}}$, and~${W' := W}$ if no such dicut exists. 
    For each non-empty proper subset~${Z \subsetneq W}$, let~$P_Z$ be a directed path from~$Z$ to~${W \setminus Z}$ in~$D'$ if such a path exist and let~$\mathcal{P}$ denote the set of all paths~$P_Z$. 
    Now let~$D^\ast$ be the digraph obtained from~${D'[W'] \cup \mathcal{P}}$ by deleting for each pair of vertices all but finitely many edges of capacity~$0$ between them. 
    
    Note that every component of~$D^\ast$ contains a vertex in~$W'$ and each edge of~$D^\ast$ is contained in a directed path between some vertices of~$W'$. 
    Consider a dicut~$B$ of~$D^\ast$ with a representation~${(X,Y)}$. 
    As each component of~$D^\ast$ contains a vertex of~$W'$, note that both~${X \cap W'}$ and~${Y \cap W'}$ are non-empty proper subsets of~$W'$. 
    If both~${X \cap W}$ and~${Y \cap W}$ are proper non-empty subsets of~$W$, then by the choice of~$\mathcal{P}$, there is a dicut~$B'$ of~$D'$ with a representation~${(X',Y')}$ such that~${X \cap W \subseteq X'}$ and~${Y \cap W \subseteq Y'}$. 
    Otherwise, one of~$X$ or~$Y$ is disjoint from~$W$ and is equal to~$\{v_0\}$, thus~$B$ has capacity~$0$ and we can choose~$B'$ to be the dicut of~$D'$ separating~${\{v_0\}}$ from~$W$. 
    In particular, in both cases we obtain~${B \setminus c^{-1}(0) = B' \setminus c^{-1}(0)}$. 
    
    On the other hand, each dicut~${B' \in \mathcal{B}'}$ with a representation~${(X',Y')}$ that separates~$W$ defines a dicut~$B$ of~$D^\ast$ represented by~${(X' \cap V(D^\ast), Y' \cap V(D^\ast))}$ for which we trivially obtain that~${B \setminus c^{-1}(0) = B \setminus c^{-1}(0)}$. 
    Lastly, if there is a dicut~${B' \in \mathcal{B}'}$ with a representation~${(X',Y')}$ that does not separates~$W$, then~${c(B') = 0}$ and~${( \{v_0\}, V(D^\ast) \setminus \{v_0\} )}$ represents a dicut~$B$ with~${B \setminus c^{-1}(0) = B' \setminus c^{-1}(0) = \emptyset}$. 
    Hence, the result follows with any finite~${\overline{\mathcal{B}} \supseteq \mathcal{B}}$ that for each~${B \in \mathcal{B}'}$ contains a~${B' \in \mathcal{B}'}$ with~${B \setminus c^{-1}(0) = B' \setminus c^{-1}(0)}$. 
\end{proof}

We now lift Theorem~\ref{thm:finite_K_5-3} to infinite digraphs using Lemmas~\ref{lem:compactness} and~\ref{lem:minorofaux2}. 

\begin{corollary}
    \label{cor:infinite_K_5-3}
    Every weakly connected digraph~$D$ 
    (with capacity~$c$) 
    whose underlying multigraph contains no minor isomorphic to either the triangular prism~$\triangularprism$, $K_5$ or~$K_{3,3}$
    is $\mathfrak{B}_\textnormal{fin}${\nbd}woodall 
    (with respect to~$c$) 
    for the class~$\mathfrak{B}_\textnormal{fin}$ of dicuts of~$D$ of finite size (capacity). 
\end{corollary}

\begin{proof}
    We may assume that~$D$ contains no dicuts of capacity~$0$ or else there is nothing to show. 
    Consider a finite set~${\mathcal{B}}$ of dicuts of~$D$ of finite capacity. 
    By Lemma~\ref{lem:minorofaux2} there is a finite set~$\overline{\mathcal{B}}$ of dicuts with~${\mathcal{B} \subseteq \overline{\mathcal{B}}}$ of finite capacity
    and a finite minor~$D^\ast$ of~$D$ such that 
    with~$\mathcal{B}^\ast$ denoting the set of dicuts of~$D^\ast$ and~$c^\ast$ denoting the capacity of~$D^\ast$ obtained from restricting~$c$ to~$E(D^\ast)$, 
    we have that~$\{ B \setminus c^{-1}(0) \colon B \in \overline{\mathcal{B}} \} = \{ B \setminus c^{-1}(0) \colon B \in \mathcal{B}^\ast \}$. 
    In particular, we conclude~${D{/}\overline{\mathcal{B}}}$ is $\overline{\mathcal{B}}${\nbd}woodall with respect to~$c$ if and only if~$D^\ast$ is~$\mathcal{B}^\ast${\nbd}woodall with respect to~$c^\ast$. 
    
    Since~$\overline{\mathcal{B}}$ contains no dicuts of capacity~$0$, we observe that~$\emptyset \notin \mathcal{B}^\ast$ and hence that~$D^\ast$ is weakly connected. 
    Since~$D$ does not contain a minor isomorphic to either~$\triangularprism$, $K_5$ or~$K_{3,3}$, neither does~$D^\ast$. 
    Therefore, $D^\ast$ is~$\mathcal{B}^\ast${\nbd}woodall with respect to~$c^\ast$ by Theorem~\ref{thm:finite_K_5-3}. 
    The result now follows from Lemma~\ref{lem:compactness}. 
\end{proof}

\medskip

Before we come to the next result we again have to introduce further notation. 

A one-way infinite path is called a \emph{ray} and the unique vertex of degree~$1$ in a ray is called its \emph{start vertex}. 
An orientation of a ray~$R$ such that every vertex is oriented away from the start vertex of~$R$ is called a \emph{forwards directed ray}, or briefly an \emph{out-ray}. 
A \emph{backwards directed ray}, or briefly a \emph{back-ray}, is defined analogously. 

For a weakly connected digraph~$D$ call a strongly connected component~$C$ of~$D$ a \emph{source component} 
if no edge of~$D$ has its head in~${V(C)}$ and its tail in~${V(D) \setminus V(C)}$. 
A \emph{sink component} of $D$ is defined analogously.
Furthermore, call a dicut~$B$ of~$D$ \emph{sink-sided} (resp.~\emph{source-sided}) if~${\outsh_D(B)}$ (resp.~${\insh_D(B)}$) contains neither a sink component (resp.~source component) of~$D$ nor a out-ray (resp.~back-ray) of~$D$. 
A dicut of~$D$ that is either source-sided or sink-sided is called a \emph{source-sink} dicut. 

The following result is due to Feofiloff and Younger. 
We state it here adapted to our notation.

\begin{theorem}
    \cite{FY:source-sink}
    \label{thm:fin_source-sink}
    Every finite weakly connected digraph~$D$ 
    (with finitary capacity~$c$) 
    is $\mathfrak{B}_\textnormal{s-s}${\nbd}woodall 
    (with respect to~$c$) 
    for the class~$\mathfrak{B}_\textnormal{s-s}$ of all source-sink dicuts of~$D$.
\end{theorem}

Let us call a weakly connected digraph~$D$ \emph{source-sink connected} if 
for every~$C^+$ which is either a source component of~$D$ or a back-ray of~$D$, 
and for every~$C^-$ which is either a sink component of~$D$ or a out-ray of~$D$, there exists a directed path from~$C^+$ to~$C^-$ in~$D$. 

With Theorem~\ref{thm:fin_source-sink}, Feofiloff and Younger verified Conjecture~\ref{conj:woodall} for the class of finite source-sink connected digraphs since each dicut of a finite source-sink connected digraph is a source-sink dicut. 

Now we shall lift Theorem~\ref{thm:fin_source-sink} to infinite graphs using Lemma~\ref{lem:compactness}. 

\begin{corollary}
    \label{cor:inf_source-sink}
    Every weakly connected digraph~$D$ 
    (with capacity~$c$)
    is $\mathfrak{B}_\textnormal{s-s}$-woodall 
    (with respect to~$c$) 
    for the class~$\mathfrak{B}_\textnormal{s-s}$ of all source-sink dicuts of~$D$ of finite size (capacity).
\end{corollary}

\begin{proof}
    Note that given a finite set ${\mathcal{B} \subseteq \mathfrak{B}_\textnormal{s-s}}$ every~${B \in \mathcal{B}}$ is also a finite source-sink dicut of~${D{/}\equiv_{\mathcal{B}}}$. 
    Hence, from Theorem~\ref{thm:fin_source-sink} we deduce that~$D{/}\equiv_{\mathcal{B}}$ is $\mathcal{B}$-woodall with respect to the capacity obtained by restricting~$c$ to~${E(D{/}\equiv_{\mathcal{B}})}$, which indeed is finitary. 
    The result now follows from Lemma~\ref{lem:compactness}. 
\end{proof}

As for finite digraphs, this has an immediate consequence for source-sink connected digraphs regarding Conjecture~\ref{conj:woodall} and the class of all finite dicuts. 

\begin{corollary}
    \label{cor:inf_source-sink_woodall}
    Every weakly connected, source-sink connected digraph $D$ 
    (with capacity~$c$) 
    is $\mathfrak{B}_\textnormal{fin}${\nbd}woodall 
    (with respect to~$c$) 
    for the class~$\mathfrak{B}_\textnormal{fin}$ of dicuts of~$D$ of finite size (capacity). 
\end{corollary}

\begin{proof}
    The proof follows from Corollary~\ref{cor:inf_source-sink} and the observation that every dicut of~$D$ is a source-sink dicut.
\end{proof}

\subsection{Classes of infinite dibonds}
\label{subsec:infinite-dibonds}
\ 

In this subsection, we will prove Theorem~\ref{thm:infinite-cardinality-woodall}. 

First we concentrate on the case where each dibond in the class has the same size as the digraph itself. 
Note that the following proof works for sets of bonds in undirected multigraphs as well. 

\begin{lemma}
    \label{lem:regular-woodall}
    Let~$\kappa$ be an infinite cardinal, 
    let~$D$ be a 
    weakly connected 
    digraph of size~$\kappa$, 
    and let~$\mathfrak{B}$ be a class of dibonds of~$D$ each of which has size~$\kappa$. 
    Then~$D$ is $\mathfrak{B}${\nbd}woodall. 
\end{lemma}

\begin{proof}
    We build the dijoins inductively.
    For each~${i < \kappa}$ we start with empty sets~$F_i^{0}$.
    We fix an arbitrary enumeration ${\{ (i_\alpha, u_\alpha, v_\alpha ) \setsep \alpha < \kappa \}}$ of the set~${\kappa \times V(D) \times V(D)}$.
    
    Suppose for~${\alpha < \kappa}$ we already constructed a family of disjoint sets~${( F_i^{\alpha} \setsep i < \kappa )}$ of edges
    such that~${F^\alpha := \bigcup \{ F_i^\alpha \setsep i < \kappa \}}$ has cardinality less than~${\lvert \alpha \rvert^+ \cdot \aleph_0}$. 
    Let~${X^{\alpha} \subseteq V(G)}$ denote the set containing the end vertices of~${F^{\alpha}}$ as well as~$u_\alpha$ and~$v_\alpha$.
    For each pair of distinct vertices~${x, y \in X^{\alpha}}$, let~${P^\alpha(x,y)}$ denote an undirected path between~$x$ and~$y$ in~$D$ which is edge disjoint to~$F^{\alpha}$ if such a path exists, or let~${P^\alpha(x,y) := \emptyset}$ otherwise. 
    We set
    \[
        {F_{i_\alpha}^{\alpha+1} := F_{i_\alpha}^{\alpha} \cup \bigcup \{ P^\alpha(x,y) \setsep x,y \in X^\alpha\}} \text{ and  } {F_{j}^{\alpha + 1} := F_{j}^{\alpha}} \text{ for each } {j \neq i_\alpha}.
    \]
    Note that~${F^{\alpha+1} = \bigcup \{ F_i^{\alpha+1} \setsep i < \kappa\}}$ has cardinality less than~${\lvert \alpha \rvert^+ \cdot \aleph_0}$, and hence we can continue the construction. 
    
    For a limit ordinal~${\lambda \leq \kappa}$, we set~${F_i^\lambda := \bigcup \{ F_i^\alpha \setsep \alpha < \lambda \}}$ for each~${i < \kappa}$. 
    Note that by construction each~$F_i^\lambda$ has cardinality at most~${\lvert \lambda \rvert}$. 
    Moreover, for all but at most~$\lambda$ many~${i < \kappa}$ the set~$F_i^\lambda$ is empty. 
    Hence~${F^\lambda := \bigcup \{ F_i^\lambda \setsep i < \kappa \}}$ has cardinality at most~${\lvert \lambda \rvert^2 = \lvert \lambda \rvert < \lvert \lambda \rvert^+}$ and we can continue the construction as long as~${\lambda < \kappa}$. 
    
    \begin{claim}
        \label{claim:construction-dijoins}
        $F_{i_\alpha}^{\alpha+1}$ meets every dibond~${B \in \mathfrak{B}}$ 
        separating~$u_\alpha$ and~$v_\alpha$.
    \end{claim}
    
    \begin{proof}
        [Proof of Claim~\ref{claim:construction-dijoins}]
        By construction~$X^\alpha$ meets both~$\outsh_D(B)$ and~$\insh_D(B)$. 
        Since~$B$ has size~$\kappa$ there is an edge in~${B \setminus F^{\alpha}}$, and since~$B$ separates some pair of vertices in~$X^\alpha$, there are vertices~${x, y \in X^\alpha}$ with~${x \in \outsh_D(B)}$ and~${y \in \insh_D(B)}$ for which there is an undirected path between~$x$ and~$y$ which is edge disjoint to~$F^{\alpha}$.
        And since every such path meets~$B$, so does~${P^\alpha(x,y) \neq \emptyset}$ and hence~$F_{i_\alpha}^{\alpha+1}$. 
    \end{proof}
    
    With Claim~\ref{claim:construction-dijoins} we can deduce that the set~${\{ F_i^\kappa \setsep i < \kappa \}}$ is the desired set of disjoint $\mathfrak{B}${\nbd}dijoins.
\end{proof}

A \emph{decomposition}~$\mathcal{H}$ of a graph~$G$ is a set of subgraphs of~$G$ such that each edge of~$G$ is contained in a unique~${H \in \mathcal{H}}$. 
For an infinite cardinal~$\kappa$, a decomposition of~$G$ is \emph{$\kappa${\nbd}bond-faithful} if 
\begin{enumerate}
    \item each~$H$ has at most~$\kappa$ many edges; 
    \item any bond of size at most~$\kappa$ of~$G$ is a bond of some~${H \in \mathcal{H}}$; and 
    \item any bond of size less than~$\kappa$ of some~${H \in \mathcal{H}}$ is a bond of~$G$. 
\end{enumerate}

\begin{theorem}
    [Laviolette \cite{laviolette}*{Theorem~3}, Soukup \cite{soukup-elementary-submodels}*{Theorem~6.3}]
    \label{thm:laviolette}
    For all infinite cardinals~$\kappa$ every graph has a $\kappa${\nbd}bond-faithful decomposition.
\end{theorem}

Note that Laviolette originally only proved this theorem under the assumption of the generalised continuum hypothesis \cite{laviolette}*{Theorem~3}. 
This assumption was subsequently removed by Soukup using the technique of elementary submodels \cite{soukup-elementary-submodels}*{Theorem~6.3}. 

Moreover, note that while this theorem was originally proven for simple graphs, it holds for multigraphs as well, 
and additionally we may assume that each graph in the decomposition is connected, 
as the following corollary summarises. 

\begin{corollary}
    \label{cor:laviolette-multigraph}
    For every infinite cardinal~$\kappa$ every multigraph has a $\kappa${\nbd}bond-faithful decomposition into connected graphs. 
\end{corollary}

\begin{proof}
    For a multigraph~$G$, consider a simple graph~$G'$ obtained by iteratively deleting parallel edges and loops. 
    By Theorem~\ref{thm:laviolette}, this graph has a $\kappa${\nbd}bond-faithful decomposition~$\mathcal{H}'$. 
    For each~${H \in \mathcal{H}'}$ which is not connected, we replace it by its connected components to obtain a decomposition~$\mathcal{H}''$, which is again $\kappa${\nbd}bond-faithful as every bond is contained in a unique connected component. 
    We construct a decomposition of~$G$ as follows. 
    For any two vertices~$v$ and~$w$ such that there are at most~$\kappa$ many parallel edges between~$v$ and~$w$, we add all of those edges to the unique simple graph~${H \in \mathcal{H}''}$ containing~$vw$. 
    Otherwise, we decompose the edges between~$v$ and~$w$ into sets of size~$\kappa$, add one of those sets to the unique simple graph~${H \in \mathcal{H}''}$ containing~$vw$, and for each other of those sets add a new graph to the decomposition consisting of precisely the edges in that set. 
    Now it is easy to verify that the decomposition~$\mathcal{H}$ obtained in this manner is $\kappa${\nbd}bond-faithful. 
\end{proof}

We will use this concept to deduce Theorem~\ref{thm:infinite-cardinality-woodall}.

\begin{proof}[Proof of Theorem~\ref{thm:infinite-cardinality-woodall}]
    Let~$\kappa$ be the cardinality of a smallest dibond in~$\mathfrak{B}$.
    Let~$\mathcal{H}$ be a $\kappa${\nbd}bond-faithful decomposition of the underlying multigraph as in Corollary~\ref{cor:laviolette-multigraph}.
    
    Every dibond~$B$ in~$\mathfrak{B}$ induces a dicut of size at most~$\kappa$ in some of the members of the $\kappa${\nbd}bond-faithful decomposition. 
    This dicut cannot contain dibonds of size less than~$\kappa$ since such dibonds would be dibonds of~$D$ contained~$B$, contradicting that~$B$ is a dibond. 
    
    To each~${H \in \mathcal{H}}$ we apply Lemma~\ref{lem:regular-woodall} to the class~$\mathfrak{B}_H$ of those dibonds of~$H$ that are contained in some dicut of~$H$ that is induced by some dibond~${B \in \mathfrak{B}}$. 
    Let~${\{ F_i^H \setsep i < \kappa \}}$ denote the set of disjoint $\mathfrak{B}_H${\nbd}dijoins of~$H$. 
    It is now easy to see that~${\left\{ \bigcup \{ F_i^H \setsep H \in \mathcal{H} \} \setsep i < \kappa \right\}}$ is a set of disjoint $\mathfrak{B}${\nbd}dijoins.
\end{proof}

Finally, we show that a generalisation of this theorem to classes of dicuts fails.
With the observations from Section~\ref{sec:capacity} we also see that a capacitated version of Theorem~\ref{thm:infinite-cardinality-woodall} fails. 

\begin{example}
    \label{ex:infinite-dicuts}
    For any infinite cardinal~$\kappa$ consider the digraph~$D$ 
    consisting of~$\kappa$ many pairwise non-incident edges~$e_\alpha$ for all~${\alpha < \kappa}$, i.e.~the edge set of~$D$ is  
    \[
        E := \{ e_\alpha \setsep \alpha < \kappa \}.
    \]
    For every~${I \subseteq \kappa}$ let~${B_I := \{ e_\alpha \setsep \alpha \in I \}}$ denote the dicut consisting of the edges index by elements from~$I$. 
    Now consider the class of dicuts 
    \[
        {\mathfrak{B} := \{ B_I \setsep I \subseteq \kappa, \, \abs{I} = \kappa \}}. 
    \]
    Note that any $\mathfrak{B}${\nbd}dijoin~$F$ of~$D$ has size at least~$\kappa$ since there are~$\kappa$ many disjoint dicuts in~$\mathfrak{B}$. 
    Moreover, note that~${E \setminus F}$ has size less than~$\kappa$ since~$B_I \notin \mathfrak{B}$ for the set~$I$ for which~${B_I = E \setminus F}$. 
    Hence,~$D$ does not contain two disjoint $\mathfrak{B}${\nbd}dijoins both contained in~$E$. 
\end{example}

However, we conjecture that the generalisation holds for nested classes of dicuts, which would yield with the observations from Section~\ref{sec:capacity} the capacitated version for nested classes of infinite dicuts.

\begin{conjecture}
    Let~$D$ be a digraph and~$\mathfrak{B}$ be a nested class of infinite dicuts of~$D$. 
    Then~$D$ is $\mathfrak{B}${\nbd}woodall. 
\end{conjecture}

\end{document}